\documentclass[a4paper,12pt,final]{article}

\usepackage[T2A]{fontenc}
\usepackage[cp1251]{inputenc}
\usepackage[english]{babel}

\usepackage{hyperref}
\usepackage{orcidlink} 

\hypersetup{
	pagebackref=true,
	colorlinks=true,
	linkcolor=blue,
	filecolor=magenta,      
	urlcolor=cyan,
	pdftitle={Numerical method for abstract Cauchy problem with nonlinear nonlocal condition},
	pdfauthor={Volodymyr Makarov, Dmytro Sytnyk, Vitalii Vasylyk}
	pdfpagemode=FullScreen,
}
\usepackage{amsmath,amsfonts,amssymb,amsthm, graphicx,array,overpic}

\usepackage[capitalize,nameinlink]{cleveref}
\usepackage[hyperpageref]{backref} 
\renewcommand*{\backrefalt}[4]{%
	\hypersetup{linkcolor=gray}%
	\color{gray}{%
		[%
		\ifcase #1 %
		No citations%
		\or
		Cited on p. #2%
		\else
		Cited on pp. #2%
		\fi
		]%
	}
}

\usepackage{xcolor}

\usepackage{subcaption}

\newtheorem{newremark}{Remark}[section]
\newenvironment{remark}
{\begin{newremark}\rm}
	{\end{newremark}}
\newtheorem{theorem}{Theorem}

\newtheorem{example}[theorem]{Example}
\newtheorem{lemma}[theorem]{Lemma}

\newcommand{\ee}{{\rm e}}
\newcommand{\R}{\mathbb{R}}
\newcommand{\oC}{\mathbb{C}}

\newcommand{\cE}{ {\cal E} }

\newcommand{\vz}{\vec{z}}
\newcommand{\vu}{\vec{u}}
\newcommand{\vy}{\vec{y}}
\newcommand{\vx}{\vec{x}}
\newcommand{\vp}{\vec{p}}

\newcommand{\vg}{\vec{g}}

\newcommand{\vertiii}[1]{{\left\vert\kern-0.25ex\left\vert\kern-0.25ex\left\vert #1 
		\right\vert\kern-0.25ex\right\vert\kern-0.25ex\right\vert}}

\begin{document}
	\selectlanguage{english}	
	\title{Numerical method for abstract Cauchy problem with nonlinear nonlocal condition}
	
	\author{%
		V.~Makarov\,\orcidlink{0000-0002-4883-6574}
		\thanks{Institute of mathematics, NAS of Ukraine,
			({\tt makarov@imath.kiev.ua}).} \and 
		D.~Sytnyk\,\orcidlink{0000-0003-3065-4921}\thanks{Institute of mathematics, NAS of Ukraine,
			({\tt sytnik@imath.kiev.ua}).} \and
		V.~Vasylyk\,\orcidlink{0000-0002-4235-9674}\thanks{Institute of mathematics, NAS of Ukraine,
			({\tt vasylyk@imath.kiev.ua}).}	}
	
	\date{\null}
	\maketitle
	
	{\small
		{\bf Abstract.} Problem for the first order differential equation with an unbounded operator coefficient in Banach space and nonlinear nonlocal condition is considered. A numerical method is proposed and justified for the solution of this problem under assumptions that the mentioned operator coefficient $A$ is strongly positive and some existence and uniqueness conditions are fulfilled. The method is based on the reduction of the given problem to an abstract Hammerstein equation. The later one is discretized by collocation and then solved via the fixed--point iteration method. Each iteration of the method involves  Sinc-based numerical evaluation of the operator exponential represented by a Dunford-Cauchy integral along hyperbola enveloping the spectrum of $A$. 
	}
	
	\section{Introduction}\label{sec:CC-NP-intro}
	The paper presents a numerical method for the 
	nonlocal initial value problem  
	\begin{equation}\label{eq:CP-NN}
	\begin{split}
	& \frac{\partial u(t)}{\partial t}+Au(t)=0, \quad t \in (-1,1),\\
	&u(-1) - g\left (u(\cdot)\right )=u_0,
	\end{split}
	\end{equation}
	where $A$ is a linear sectorial operator with the dense domain $D(A)$ in Banach space $X$ and $g: C\left ([-1,1]; X\right ) \rightarrow X$ is a given operator function. 
	Recall that the sectorialness of $A$ means that its spectrum  lies inside a sectorial domain $\Sigma \in \oC$, 
	\begin{equation}\label{spectrA}
	\Sigma\left (\rho, \varphi \right ) = \left\{ z=\rho+r \ee^{i\theta}:\ r \in  [0,\infty), \, \rho \in \R_{+}, \, \left|\theta\right|< \varphi
	\right\},
	\end{equation}
	and the resolvent $R_A(z) = \left (zI - A \right )^{-1}$ of $A$ satisfies the inequality
	\begin{equation}\label{estrez}
	\lVert R_A(z)\rVert \leq \frac{M}{1+\left |z\right|},
	\end{equation}
	on the boundary of $\Sigma$ and outside it.
	A unique pair of parameters $\left (\rho_0, \varphi_0\right )$ such that 
	$\Sigma\left (\rho_0, \varphi_0\right ) \subseteq \Sigma\left (\rho, \varphi \right )$, for any other admissible pairs $(\rho, \varphi )$,
	%
	is called spectral parameters of $A$. 
	
	The aforementioned nonlocal problem \eqref{eq:CP-NN} is well studied from the theoretical point of view.  Initial investigations into the existence of solution to \eqref{eq:CP-NN} is given in the works of Byszewski \cite{Byszewski1991a, Byszewski1992} for the case when $g$ is a function of the finite number of $u(t_i)$, $i=\overline{1,n}$. 
	Further developments and generalizations can be found in \cite{Byszewski1992a, NtouyasTsamatos1997, AizicoviciLee2005, XiaoLiang2005, nonloc_exMVS2014} (see also the review in \cite{Ntouyas2006461}).
	In this work we will use the existence conditions obtained in \cite{FanDongLi2006}.
	
	Available works on numerical methods for the nonlocal problems of type \eqref{eq:CP-NN} usually deals with the particular realizations of operator $A$ and nonlocal condition specific to some application area. Pioneering work of Bitsadze and Samarskii \cite{NonlocalBitsadze1969}, inspired by the plasma research, deals with two-point nonlocal problem for elliptic partial differential operators.  Similar two and three point nonlocal conditions appear in the works of Gordeziani \cite{nonloc_BS_Gordeziani1970, Gordeziani1981, Gordeziani1984} in applications two the dynamics of the thin-walled structures. In \cite{nonlocal_time_numerVabishchevich1982} Vabishchevich investigated the application of two-point nonlocal problems to the inverse problems of heat conduction (see also \cite{AshGer,nonloc_int_time_Ashyralyev2014}). Galerkin methods for multipoint nonlocal problems related to the diffusion processes  were developed in the works of Canon with co-authors \cite{Cannon1986, Cannon1987}.
	
	Majority of the existing numerical methods are based on the finite difference or finite element approximations to the given differential equation.
	More novel numerical approaches presented in \cite{Gavrilyuk2010,vasNL_Ell,VasIntEll} are based on the direct approximation of the solution  operator employing the Dunford-Cauchy formula and quadrature specifically adjusted to the spectral parameters of $A$.  
	The later technique permits one to obtain natively parallelizable  numerical methods with the accuracy automatically adjustable to the smoothness of initial data (methods without accuracy saturation).
	These two numerical properties are essential for the modern problem driven scientific simulations using the state-of-the-art multi-core computational architectures. Hence, we set the development of such a method for \eqref{eq:CP-NN} as the main goal of current work.
	
	Section \ref{sec:CC-NP-int_repr} serves a preparatory purpose. In this section the original problem is transformed to a more general  nonlinear representation. 
	In section \ref{sec:CP-NN-disctretization} we show how to discretize the obtained nonlinear Hammerstein equation. 
	The discretized problem is then studied in section \ref{sec:CC-NN-sol_disc_sys}, where we establish conditions on the existence of its solution with help of the Banach fixed point theorem. The accuracy estimate for the error of fixed point iteration method is obtained therein. 
	Final error estimate of the method is derived in section \ref{sec:CP-NN-final_error}. 
	The remaining part of this section is devoted to a numerical example, which demonstrates the method's effectiveness.
	Concluding remarks and possible extension are given in section  \ref{sec:CC-NN-conclusions}.

	\section{Alternative representation of given problem}\label{sec:CC-NP-int_repr}
	The Hille-Yosida theorem ensures the existence of strongly continuous semigroup $\ee^{-At}$,  if $A$ is a sectorial operator with the dense domain $D(A) \subseteq X$. 
	Furthermore it is well known from the classical semigroup theory \cite{bPazy1983} that for such $A$ the general solution to the abstract differential equation from \eqref{eq:CP-NN} can be formally written as 
	\[
	u(t) = T(A,t)u(-1) \equiv \ee^{-A(t+1)}u(-1), \quad \forall  t \in [-1,\,1], \ u(-1) \in X. 
	\]
	When $u(-1) \notin D(A)$, the action $T(A,t)u(-1)$ should be understood in a sense of the limit. 
	Next we incorporate the nonlocal condition given in \eqref{eq:CP-NN} by substituting the expression for $u(-1)$ therefrom. It results in the general nonlinear problem 
	\begin{equation}\label{eq:CP-NN-ham_eq}
	u(t)=T(A,t)\left[u_0 + g\left (u(\cdot)\right )\right].
	\end{equation}
	Equation \eqref{eq:CP-NN-ham_eq} 
	is commonly known as the abstract Hammerstein equation \cite{NHE_Hess1971, Zeidler1990}. 
	Unlike \eqref{eq:CP-NN} this equation is valid for any $u_0 \in X$, and  becomes equivalent to \eqref{eq:CP-NN}, when $u(t)$ is differentiable and $u_0 \in D(A)$. 
	The function satisfying \eqref{eq:CP-NN-ham_eq} is called a mild solution of \eqref{eq:CP-NN}, while the original solution of \eqref{eq:CP-NN} is called strong. 
	Note that, by derivation of \eqref{eq:CP-NN-ham_eq}, both equations from \eqref{eq:CP-NN} are incorporated into one formula, ready-made for the use of fixed point iteration.
	For that reason, either Banach \cite{Byszewski1991a} or Schauder \cite{NtouyasTsamatos1997} fixed point theorem can be directly applied to \eqref{eq:CP-NN-ham_eq} resulting in the existence conditions for the mild solution 
	(see \cite{NtouyasTsamatos1997, Ntouyas2006461} for the detailed reviews of applicable techniques). 
	
	At first, it seems that equation \eqref{eq:CP-NN-ham_eq} is not so valuable from the computational point of view, since it contains operator function  $T(A,t)$. Its numerical evaluation is a non-trivial computationally involving task even in the case when $A$ is matrix  \cite{MolerLoan2003}. 

	In reality stable and efficient approximation of $T(A,t)$ is possible for $t \in [-1,\,1]$, as long as $A$ is a linear sectorial operator with $\varphi < \frac{\pi}{2}$ \cite{GavrilyukMakarovVasylyk2011}. 
	Next we will summarize how to build the exponentially convergent approximation of $T(A,t)v$ for $v \in D(A^\alpha)$, $\alpha >0$. 
	The action of operator exponential $T(A,t)$ on $v$ admits the following integral representation  \cite{bClement1987, bPazy1983}
	\begin{equation}\label{eqDunfordCauchyRepr}
	\ee^{-A(t+1)}v=\frac{1}{2\pi i}\int_{\Gamma}\ee^{-z(t+1)}R_A(z)vdz,
	\end{equation}
	here $R_A(z)=(z I-A)^{-1}$ is the resolvent of $A$,  and $\Gamma \in \oC\setminus\Sigma $ stands for the integration contour, positively oriented  with respect to the spectrum $\Sigma$ of $A$. According to \cite{GavrilyukMakarovVasylyk2011} the integration contour for the Dunford-Cauchy integral in \eqref{eqDunfordCauchyRepr} needs to be chosen as 
	\[
	\Gamma=\left\lbrace z(s)=a_I\cosh(s)-ib_I\sinh(s):\, s\in (-\infty,\infty) \right\rbrace. 
	\]
	The unknown values of contour parameters $a_I$, $b_I$ are to be determined in such a way that the integrand can be analytically extended into the strip $D_d \in \oC$ 
	$$
	D_d = \left\{\xi=x+i y\, \middle|\, x \in \R, |y| < \frac{d}{2} \right\},
	$$
	and remains bounded there.
	For a fixed $A$ the strip width parameter $0< d <\frac{\pi}{2}$ is responsible for the accuracy of the sinc-quadrature used below to approximate \eqref{eqDunfordCauchyRepr}.
	The mentioned Sinc-quadrature provides exponentially convergent approximation of $T(A,t),\, t>-1$ and attains its fastest convergence rate,  when 
	\[
	d=\frac{\pi}{2} -\varphi.
	\]
	For such $d$ and $(\rho_0,\varphi_0)$ the values of $a_I$, $b_I$ are specified by the formulas
	
	\begin{equation}\label{eq:CP-NN_aI_bI}
	a_I= \rho_0 \frac{\cos{\left(\frac{d}{2} +\varphi_0\right)}}{\cos{\varphi_0}} , \quad b_I= \rho_0 \frac{\sin{\left(\frac{d}{2} +\varphi_0\right)}}{\cos{\varphi_0}}.
	\end{equation}
	In addition to the performed parametrization of $\Gamma$ we make the specific resolvent correction to \eqref{eqDunfordCauchyRepr} by putting
	\[
	R_{A,1}(z)=R_A(z)-\frac{1}{z}I
	\]
	in place of $R_{A}(z)$. 
	Such correction compensates the poor decay of $T(A,-1)$ at infinity, and allows to guaranty the exponential convergence rate of the mentioned Sinc-quadrature \cite{GavrilyukMakarovVasylyk2011}. 
	After we conduct the described manipulations on \eqref{eq:CP-NN-ham_eq} it will take the form 
	\begin{equation}\label{eq:CP-NN-intrivDC}
	\begin{split}
	u(t)= &\int\limits_{-\infty}^{\infty} e^{-z(\xi)\left(t+1\right)} z'(\xi) R_{A,1}\left (z\left (\xi\right )\right )\left[u_0 + g\left (u(\cdot)\right )\right] d\xi,\\
	z^\prime(\zeta)=& a_I \sinh{\zeta}-i b_I \cosh{\zeta}.
	\end{split}
	\end{equation}
	\begin{remark}
		The reader might have noted that the performed correction can only be justified if the spectral shift $\rho_0$ is positive. Indeed for the negative $\rho_0$ both the correction and the definition of $a_I, b_I$ have to be modified. On the other hand, one might always make spectral shift to be greater than zero by a simple transformation 
		\[
		v(t) = e^{\rho_1 t} u(t), \quad \rho_1>0.
		\]  
	\end{remark}
	
	\section{Discretization}\label{sec:CP-NN-disctretization}
	The next step toward the fully discretized analogue of \eqref{eq:CP-NN-ham_eq}, relies upon the approximation of the integral in \eqref{eq:CP-NN-intrivDC}.
	In this section we utilize the quadrature based approximation developed in authors' earlier works (see \cite{GavrilyukMakarovVasylyk2011} and the references therein).  The operator exponential $\ee^{-A(t+1)}v$ is approximated by $T_N(A,t)v$, 
	\begin{equation}\label{eq:op_exp_approx}
	T_N(A,t)v  = \frac{h}{2\pi i}\sum_{p=-N}^N\ee^{-z(ph)(t+1)}z^{\prime}(ph)R_{A,1}(z(ph))v.	
	\end{equation}
	Expression on the right of \eqref{eq:op_exp_approx} is obtained from the parametrized version of Dunford-Cauchy integral \eqref{eq:CP-NN-intrivDC}, with help of the Sinc-quadrature (trapezoid) formula. 
	Before we state the result concerning the accuracy of the above approximation, let us highlight that all the summands in \eqref{eq:op_exp_approx} are mutually independent.  
	The evaluation of every summand involves the calculation of the time-dependent scalar part and the evaluation of resolvent part, free of $t$.
	Computationally it means that the resolvent evaluations $R_{A,1}(z(ph))v$, $p =\overline{-N,N}$ can be performed in parallel and the results are stored. Once that is done all the evaluation  of $T_N(A,t)v$ -- for as many $t \in [-1,1]$ as needed -- can be achieved at the fraction of cost, spent for the resolvent evaluations. This is especially true if the evaluation of $R_{A,1}(z(ph))v$  predominates 
	the calculation of the scalar part in terms of computational complexity.   
	Some additional savings of computational resources are possible if the operator $A$ is real-valued \cite{Gavrilyuk2010}. 
	The accuracy of the proposed approximation is characterized by the following theorem  \cite[p. 34]{GavrilyukMakarovVasylyk2011}.
	
	\begin{theorem}\label{thm:ope_exp_est}
		Assume that $A$ is a linear sectorial operator with the densely defined domain and $v\in D(A^\alpha)$. Let 
		\begin{equation*}
		h=\sqrt{\frac{ \pi d}{\alpha(N+1)}},
		\end{equation*}		
		then the error $\eta_N(t)v \equiv \|T(A,t)v - T_N(A,t)v\|$ 
		satisfies the estimate
		\begin{equation}\label{PohEks}
		\left\| \eta_N(t)v \right\|=\left\| \ee^{-A(t+1)}v- T_N(A,t) v \right\| \le \frac{c\, \ee^{-\sqrt{\pi d \alpha(N+1)}}}{\alpha}\|A^{\alpha}v\|,
		\end{equation}
		with some positive constant $c$ independent on $A$, $v$, $\alpha$ and $t$. 
	\end{theorem} 
	Estimate \eqref{PohEks} demonstrates that the approximant $T_N(A,t)v$ meets all the requirements formulated in section \ref{sec:CC-NP-int_repr}. 
	Bearing that in mind we proceed to the collocation of \eqref{eq:CP-NN-ham_eq}. 
	More precisely we will apply the polynomial collocation method to the modified version thereof 
	\begin{equation}\label{eq:CP-NN-ham_eq_approx}
	u(t) = T_N(A,t) \left[u_0 + g\left (u(\cdot)\right )\right].
	\end{equation}
	Let us introduce the Chebyshev-Gauss-Lobatto (CGL) nodes $t_j=-\cos\left( \frac{\pi j}{n}\right)$, $j=\overline{0,n}$. For a given vector $\vy=(y_0,\dots,y_n)$ we define  the  modified Hermite-Fej\'er polynomial \cite{Smith2006} 
	\[
	K_{2n-1}(t,\vy) = \sum_{i=0}^{n}B_{i,2n-1}(t)y_i,
	\]
	of the degree $2n-1$.
	For each $i=0,1,\dots n$,  $B_{i,2n-1}$ is the unique polynomial such that  $B_{i,2n-1}(t_j) = \delta_{i,j}$, for $j=\overline{0,n}$, and $B'_{i,2n-1}(t_j) = 0$, for  $j=\overline{1,n-1}$
	\[
	\begin{split}
	B_{0,2n-1}(t) = \frac{1+t}{2} P_{n-1}^2 (t), \quad B_{0,2n-1}(t) = \frac{1-t}{2} P_{n-1}^2 (t)\\
	B_{i,2n-1}(t) = \frac{\left (1-t^2\right )\left (1 + t t_i - 2t_i^2
		\right )}{\left (n-1\right )^2\left (t - t_i\right )^2\left (1 - t_i^2\right )} P_{n-1}^2 (t), \quad i = \overline{1,n-1},
	\end{split}
	\]
	where $P_n(t)$ is Chebyshev polynomial of the first kind.
	Now, we put the polynomial $K_{2n-1}(t, \vy)$ in place of $u(t)$ in \eqref{eq:CP-NN-ham_eq_approx} and collocate the received equation at the sequence of interpolation points. It leads us to the system of nonlinear equations
	\begin{equation}\label{eq:CP-NN-disc_sys}
	y_i=T_N(A,t_i)u_0+ T_N(A,t_i)g\left (K_{2n-1}\left (\cdot, \vy\right )\right ), \quad i =\overline{0,n},
	\end{equation}
	with respect to the unknowns $y_i$. Similarly to \eqref{eq:CP-NN-ham_eq} this system, can be directly used to find the approximation to $u(t)$ on the chosen grid, since $y_i=u(t_i)$ is clearly the solution. In the following section we are going to theoretically justify the iterative solution method based upon \eqref{eq:CP-NN-disc_sys}.
	
	
	\section{Solution of discretized problem}\label{sec:CC-NN-sol_disc_sys}
	In order to prove the existence of solution to \eqref{eq:CP-NN-disc_sys} we recast this system in a vector-matrix form 
	
	\begin{equation}\label{eq:CP-NN-ham_eq_vec_it}
	\vy=\vg\left (\vy \right ) +\vp,
	\end{equation}
	where 
	\begin{equation*}
	\begin{array}{rl}
	\vg(\vy)= &\left (T_N(A,t_0)g(K_{2n-1}(\cdot,\vy)),\dots , T_N\left (A,t_n\right ) g\left (K_{2n-1}\left (\cdot,\vy\right )\right )\right )^T,  \\ 
	\vp= &(T_N(A,t_0)u_0,\dots , T_N(A,t_n)u_0)^T. 
	\end{array}
	\end{equation*}
	For the existence of the solution it is sufficient to show that a recurrence sequence
	\begin{equation}\label{eq:kapprox}
	\vy^{(k)}=\vg\left (\vy^{(k-1)}\right ) +\vp, \quad \vy^{(0)}=\vp,
	\end{equation}
	is convergent in the vector space $X^n = X \times X \times \dots X$. 		
	For any $ \vx \in X^n$ let us introduce  a norm 
	\[
	\vertiii{\vx}=\max_{0\le j \le n}\left\| x_j \right\|.
	\]
	Regarding the function $g$ we require it to satisfy the following Lipschitz-like condition for any $u,v \in C\left ([-1,1]; X\right )$
	\begin{equation}\label{eq:CP-NN-lip_cond}
	\left\| A^{\alpha}\left( g\left (u\right )-g\left (v\right ) \right) \right\| \le L \max\limits_{t \in [-1,1]}\left\| u(t)-v(t)\right\|,
	\end{equation}
	with some positive constants $\alpha$ and $L<\infty$.
	Apart from \eqref{eq:CP-NN-lip_cond} we will use the estimate 
	\begin{equation}\label{eq:approx_est}
	\left\| T_N(A,t_i)A^{-\alpha} \right\|\le \frac{c}{\alpha},
	\end{equation}
	which is a mere consequence of (3.278) from  \cite{GavrilyukMakarovVasylyk2011}. 
	\begin{theorem}\label{thm:CP-NN-sol_ex_ham_disc}
		Assume that $A:X\rightarrow X$ is a linear operator satisfying the condition of Theorem \ref{thm:ope_exp_est} 
		and $g: C\left ([-1,1]; X\right ) \rightarrow X$ is an operator function satisfying \eqref{eq:CP-NN-lip_cond}.  
		If there exist such $L,c,\alpha>0$ from \eqref{eq:CP-NN-lip_cond},\eqref{eq:approx_est}, that 
		\begin{equation}\label{UmZb}
		q\equiv\frac{3Lc}{\alpha}<1,
		\end{equation}
		then equation \eqref{eq:CP-NN-ham_eq} has a unique solution $\vy^{(\infty)}=\lim\limits_{k\rightarrow \infty} \vy^{(k)}$. Moreover 
		\begin{equation}\label{eq:est_it_sol}
		\vertiii{\vy^{(\infty)}}\le \vertiii{\vp} + \vertiii{\vg\left (\vp\right )}\frac{1}{1-q},
		\end{equation}
		and an error of the $k$-th iterative approximation admits the estimate
		\begin{equation}\label{eq:est_k_it_aprox}
		\vertiii{\vy^{(\infty)} - \vy^{(k)}}\le \vertiii{\vg\left (\vp\right )} \frac{q^{k+1}}{1-q}.
		\end{equation}
	\end{theorem}
	\begin{proof}
		To show that $y^{(\infty)}$ is a unique solution of  \eqref{eq:CP-NN-ham_eq_vec_it} we apply the Banach fixed point theorem. The space $X^n$ equipped with the metric $d(x,y) = \vertiii{\vx - \vy}$ forms a complete Banach space. 
		The mapping $\mathcal{F}$ defined by \eqref{eq:kapprox} transforms the space $X^n$ into itself. To demonstrate existence of the fixed point it remains to show that this mapping is contractive.
		\begin{align*}
		\vertiii{\mathcal{F}\vx - \mathcal{F}\vy} =&\vertiii{\vg\left (\vx\right ) - \vg\left (\vy\right ) }\\
		\leq& \max_{0\le j \le n}\left\| T_N(A,t_j)A^{-\alpha}\right\| \left \|A^{\alpha}\left ( g\left (K_{2n-1}\left (\cdot, \vx\right )\right ) - g\left (K_{2n-1}\left (\cdot, \vy\right )\right )\right )  \right\| \\
		\leq & \frac{c}{\alpha} \left \|A^{\alpha}\left ( g\left (K_{2n-1}\left (\cdot, \vx\right )\right ) - g\left (K_{2n-1}\left (\cdot, \vy\right )\right )\right )  \right\|
		\end{align*}
		We used \eqref{eq:approx_est}, to get the last estimate. For any $\vx \in X^n$ the polynomial $K_{2n-1}\left (t, \vx \right )$ is, obviously, a continuous function of $t$. Moreover, it was proved in \cite{Smith2006}, that 
		\[
		\max\limits_{t \in [-1,1]}\sum_{i=0}^{n}\left |B_{i,2n-1}(t)\right | < 3, \quad n=0,1,2,\ldots 
		\]
		These two observations together permit us to write
		\begin{equation}\label{eq:CP-NN_discsys_contraction}
		\vertiii{\mathcal{F}\vx - \mathcal{F}\vy} \leq  \frac{Lc}{\alpha}\max\limits_{t \in [-1,1]}\sum_{i=0}^{n}\left |B_{i,2n-1}(t)\right | \left \|x_i - y_i \right \| \leq \frac{3Lc}{\alpha}\vertiii{\vx - \vy}.
		\end{equation}
		The contraction of $\mathcal{F}$ is proved. 
		
		Since $\vy^{(k)}=\sum_{l=1}^{k}(\vy^{(l)}-\vy^{(l-1)}) + \vp$, we are actually interested in the difference of two consecutive elements of the sequence generated by \eqref{eq:kapprox}.  
		The estimate for $\vertiii{\vy^{(k)}-\vy^{(k-1)}}$ is provided by \eqref{eq:CP-NN_discsys_contraction}, specifically
		\begin{equation}\label{eq:it_est_dif}
		\vertiii{\vy^{(k)}-\vy^{(k-1)}} =\vertiii{\vg\left (\vy^{(k-1)}\right ) - \vg\left (\vy^{(k-2)}\right ) }
		\leq  q \vertiii{\vy^{(k-1)} - \vy^{(k-2)}}.
		\end{equation}
		So, in the end, it all goes down to 
		\[
		\vertiii{\vy^{(1)}-\vy^{(0)}} = \vertiii{\vg\left (\vp\right )} 
		\] 
		To justify \eqref{eq:est_k_it_aprox} we apply \eqref{eq:it_est_dif} to every term in the series for $\vy^{(k)}$. 
		\begin{equation}\label{eq:est2}
		\vertiii{y^{(k)}}\le \vertiii{\vp} + \sum_{p=0}^{k}\left( \frac{3Lc}{\alpha}\right) ^p \vertiii{\vg\left (\vp\right )}  = \vertiii{\vp} + \vertiii{\vg\left (\vp\right )}\frac{1-q^{k+1}}{1-q}.
		\end{equation}
		Inequality \eqref{eq:est_it_sol} can be derived analogously to \eqref{eq:est2}.
	\end{proof}

	\section{Error analysis}\label{sec:CP-NN-final_error}
	The developed method emerges as a combination of three different numerical procedures: 
	spectral approximation of the operator exponential $T_N(A,t)$, 
	collocation--based discretization and iterative solution of the discretized problem.
	In this section we derive the compound error estimate by analysing the error contribution from the every listed numerical procedure. 
	
	Let $z_i=u(t_i)-y_i$, $i=\overline{0,n}$ is the pointwise difference of solutions to \eqref{eq:CP-NN-ham_eq} and \eqref{eq:CP-NN-ham_eq_vec_it}. Denote $\vu = \left (u(t_0), \ldots, u(t_n)\right )^{T}$, $\vz = \left (z_0, \ldots, z_n\right )^{T}$. 
	Then, evaluate the quantity $z_i$ using \eqref{eq:CP-NN-ham_eq} and \eqref{eq:CP-NN-disc_sys}
	\begin{align*}
	z_i =& T(A,t_i)u_0 - T_N(A,t_i)u_0 + T(A,t_i)g\left (u(\cdot)\right ) - T_N(A,t_i)g\left (K_{2n-1} \left (\cdot, \vy\right )\right )\\
	=& T(A,t_i)u_0 - T_N(A,t_i)u_0 + T(A,t_i)g\left (u(\cdot)\right ) -  T(A,t_i)g\left (K_{2n-1} \left (\cdot, \vy\right )\right ) \\
	&+T(A,t_i)g\left (K_{2n-1} \left (\cdot, \vy\right )\right ) - T_N(A,t_i)g\left (K_{2n-1} \left (\cdot, \vy\right )\right )\\
	=& \eta_N(t_i)u_0 + \eta_N(t_i)g\left (K_{2n-1} \left (\cdot, \vy\right )\right )  + T(A,t_i)\left [g\left (u(\cdot)\right ) - g\left (K_{2n-1} \left (\cdot, \vy\right )\right )\right ]
	\end{align*}
	The terms $\eta_N(t_i)u_0$, $\eta_N(t_i)g\left (K_{2n-1} \left (\cdot, \vy\right )\right )$ can be estimated by \eqref{PohEks}, provided that the conditions of Theorem \ref{thm:ope_exp_est} are fulfilled. For convenience we denote
	\begin{align*}
	\mu_{n} &= T(A,t_i)\left [g\left (u(\cdot)\right ) - g\left (K_{2n-1} \left (\cdot, \vy\right )\right )\right ],
	\end{align*}
	and then estimate the third term 
	\[
	\begin{split}
	\vertiii{\mu_{n} } =&\max_{0\le i\le n}\left\| T_N(A,t_i)\left [g\left (u(\cdot)\right ) - g\left (K_{2n-1} \left (\cdot, \vy\right )\right )\right ] \right\| \\
	\leq& \max_{0\le i \le n}\left\| T_N(A,t_i)A^{-\alpha}\right\| 
	\left\| 
	A^{\alpha} \left[ 
	g\left( u(\cdot )\right) - g\left( K_{2n-1}\left(\cdot, \vy\right) \right) 
	\right] 
	\right\| \\
	\le& \frac{Lc}{\alpha} \max\limits_{t \in [-1, 1]} \left \|u(\cdot) - K_{2n-1} \left (\cdot, \vy\right )  \right\| \\
	\le& \frac{Lc}{\alpha} \left ( \Lambda_{2n-1}(u) + 3 \vertiii{\vec{z}}\right ) 
	\end{split}
	\] 
	where $\Lambda_{2n-1}(u)$ is an error of the approximation  of $u$ by the modified Hermit-Fej\'er polynomial of  degree $2n-1$. It is well know that, unlike Lagrange interpolation, the Hermite-Fej\'er interpolation is convergent 
	for any $u \in C\left ([-1,1]; X\right )$ \cite{Smith2006}. 
	The resulting estimate of $\vertiii{z}$ is given by the lemma below. 
	
	\begin{lemma}
		Assume that  the conditions of Theorem \ref{thm:CP-NN-sol_ex_ham_disc} is met and, in addition, that there exist $\alpha>0$ such that $u_0, g\left (K_{2n-1} \left (\cdot, \vy\right )\right ) \in D(A^\alpha)$, $\forall \vy \in X^n$.  
		Then the solution to discretized system \eqref{eq:CP-NN-disc_sys} gives a pointwise approximation of the solution to Hammerstein equation \eqref{eq:CP-NN-ham_eq}. The approximation error satisfies the following estimate 
		\begin{equation}\label{eq:CP-NN-disc_sys_err_est}
		\begin{aligned}
		\vertiii{ z} \leq& \frac{1}{1-q}\left (\frac{c\, \ee^{-\sqrt{\pi d \alpha(N+1)}}}{\alpha} \left( \|A^{\alpha}u_0\| +\left \|A^{\alpha} g\left (K_{2n-1} \left (\cdot, \vy\right )\right ) \right\|\right) \right. \\
		&\hspace*{17em}+ \left. \vphantom{\frac{c\, \ee^{-\sqrt{\pi d \alpha(N+1)}}}{\alpha}}
		    \frac{Lc}{\alpha} \Lambda_{2n-1}(u)\right ).
		\end{aligned}
		\end{equation}
	\end{lemma}
	\begin{proof}
		Most of the proof was presented before the lemma. To prove the final error estimate we combine the estimate for $\vertiii{\mu_{n} }$ with the results of Theorem \ref{thm:ope_exp_est} for  $\eta_N(t_i)u_0$, $\eta_N(t_i)g\left (K_{2n-1} \left (\cdot, \vy\right )\right )$, $i=\overline{0,n}$. 
	\end{proof}
	
	The compound error estimate of the method is given by the following theorem.
	\begin{theorem}\label{thm:CP-NN-final_err_est}
		Assume that the operator $A$ and function $g$ satisfy the conditions of Theorem \ref{thm:CP-NN-sol_ex_ham_disc} and, in addition, that 
		$u_0, g\left (K_{2n-1} \left (\cdot, \vy\right )\right ) \in D(A^\alpha)$, $\forall \vy \in X^n$. 
		Then the $k$-th iterative approximation given by \eqref{eq:kapprox}  
		constitutes a pointwise approximation of the solution to Hammerstein equation \eqref{eq:CP-NN-ham_eq}. The  error of this approximation satisfies the inequality 
		\begin{equation}\label{eq:CP-NN-final_est}
		\vertiii{\vu -\vy^{(k)}} \leq \vertiii{\vz} + c_1 M_C\frac{q^{k+1}}{1-q} \vertiii{\vp} ,
		\end{equation}
		where $M_C=\sup\limits_{v \in C\left ([-1,1]; X\right )}\frac{\|g(v)\|}{\|v\|}$, and $c_1>0$ is independent of $g$.
		
	\end{theorem}
	\begin{remark}
		Inequality  \eqref{eq:CP-NN-final_est} along with the known estimates of $\Lambda_{2n-1}(u)$ \cite{err_est_review_Goodenough1986} demonstrate that the order of the error of Hermite-Fej\'er interpolation is lower, then other contributions to the compound error estimate. To improve the overall convergence one might use other interpolation technique to collocate \eqref{eq:CP-NN-ham_eq_approx}, like spline interpolation of odd degree \cite{MakarovKhlobystov1983}, for instance. The compound numerical method could adapt any other interpolation operator with the bounded norm, as long as its image is in $C\left ([-1,1]; X\right )$.
	\end{remark}	
	 In spite of the fact that $\Lambda_{2n-1}(u) \sim \ln{n}/n$ as  $n\rightarrow \infty$ it does not present a real challenge from the computational point of view, since the computational complexity of the right-hand side evaluation from \eqref{eq:kapprox} grows very slowly when $n$ increases. It happens because of the mentioned in section \ref{sec:CP-NN-disctretization} computational properties of the operator exponential approximation.  
	\begin{example}
		Here we experimentally consider the specification of problem \eqref{eq:CP-NN}, with $A$ being a second-order elliptic operator:
		\[
		Au = -\frac{\partial^2 u}{\partial x^2}, \quad u(0) = u(1) = 0,
		\]
		and such, that the nonlocal condition is defined by
		\begin{equation}\label{cde1-1-3}
		\begin{split}
		&u(-1) - \mu \int_{-1}^{1}u^2(s)ds=u_0,\\
		&u_0(x) = \sin \left( \pi \,x \right) +\mu \frac{e^{-4\,{\pi }^{2}}-1}{2\pi^2} 
		\sin^2 \left( \pi \,x \right).
		\end{split}
		\end{equation}
		\begin{figure}
			\centering
			\includegraphics[width=0.95\linewidth]{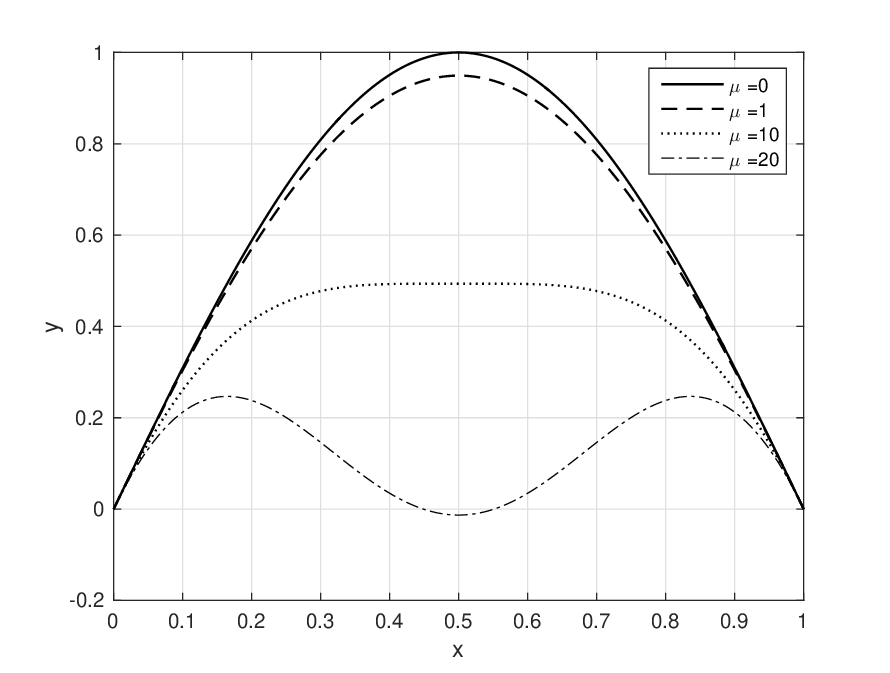}
			\caption{Graph of $u_0(x)$ for the different values of $\mu$.}
			\label{fig:u_0}
		\end{figure}
		To benchmark the developed numerical method for the different values of $q$ \eqref{UmZb} we have set $g\left (u(\cdot)\right ) = \mu \int_{-1}^{1}u^2(s)ds$. This expression contains one additional parameter $\mu$, which enable us to change the size of $q$ through $L$. 
		Exact solution $u_{ex}$ to the considered problem can be written as 
		\begin{equation}
		u_{ex}\left (t,x \right ) = {{e}^{- \left( t+1 \right) {\pi }^{2}}}\sin \left( \pi \,x
		\right). 
		\end{equation} 	
		Note, that $u_{ex}\left (t,x \right )$ does not depend on $\mu$ and is equal to $u_0(x)$, when $t=-1$, $\mu=0$.
		One can observe from the graph of $u_0$, depicted in Fig. \ref{fig:u_0}, that the solution of nonlocal problem is close to the solution of the corresponding classical Cauchy problem (with the initial condition $u(-1)=u_0$) the difference between these two solutions grows if $\mu$ gets bigger (compare the graphs of $u_0(x)$ for $\mu=0$ and $\mu=20$ from Fig. \ref{fig:u_0}).
		
		To measure an experimental accuracy of the method we define 
		\[
		Err=\mathrm{Err}\left (u_{ex},\vy^{(k)}\right ) \equiv \max_{0\le l \le m}\max_{0\le j \le n}\left\| u_{ex}(t_i, x_l) - y^{(k)}_{j}(x_l) \right\| , 
		\]
		where $x_l =\frac{1}{2}\left( 1-\cos\left( \frac{\pi l}{m}\right)\right )$, $l=\overline{0,m}$ is the scaled version of CGL nodes. After this is done, one needs to define the stopping criteria for iterative method \eqref{eq:kapprox}. 
		For that matter we shall use 
		\[
		\mathrm{Err}\left (\vy^{(k)},\vy^{(k-1)}\right ) < 10^{-18},
		\] 
		to factor out the influence of the iterative error.

		Initially, we set $\mu = 1/4$. Such choice of $\mu$ guaranties the validity of \eqref{UmZb}, for $\forall \rho_0 \in \R_+$, $\alpha =1$. 
		Moreover, function $g(\cdot)$ admits, as a function of scalar variable, analytic extension into the region $\cE_\rho \in \oC$. It means that all the suppositions of Theorems \ref{thm:CP-NN-sol_ex_ham_disc} and \ref{thm:CP-NN-final_err_est} are met.  Consequently the iterative solution of \eqref{eq:CP-NN-ham_eq_vec_it} exists and can be approximated by $\vy^{(k)}$.  This solution converges to the initial problem's solution $u(t,x)$ ($n,N \rightarrow \infty$).  
		The experimental results, calculated\footnote{The code to reproduce the experimental results is available at \url{https://www.imath.kiev.ua/~sytnik/research/works/nonlinear-nonlocal-2016.html}} by \eqref{eq:kapprox}, \eqref{eq:op_exp_approx}, a presented in Table \ref{tab1} for the different values of $N$, $n$.
		For each $N \in \{4,8,16,32,64,128,256\}$, we experimentally selected $n$ sufficiently large for the error $\mathrm{Err}(n)$ to saturate (see Fig. \ref{fig:err_vs_n} (a)). 
		\begin{figure}[ptbh]
		\centering
		\newdimen\lentwosubfig
		\lentwosubfig=0.51\linewidth
		\begin{overpic}[width=\lentwosubfig, viewport=20 20 465 380, clip=true]%
		{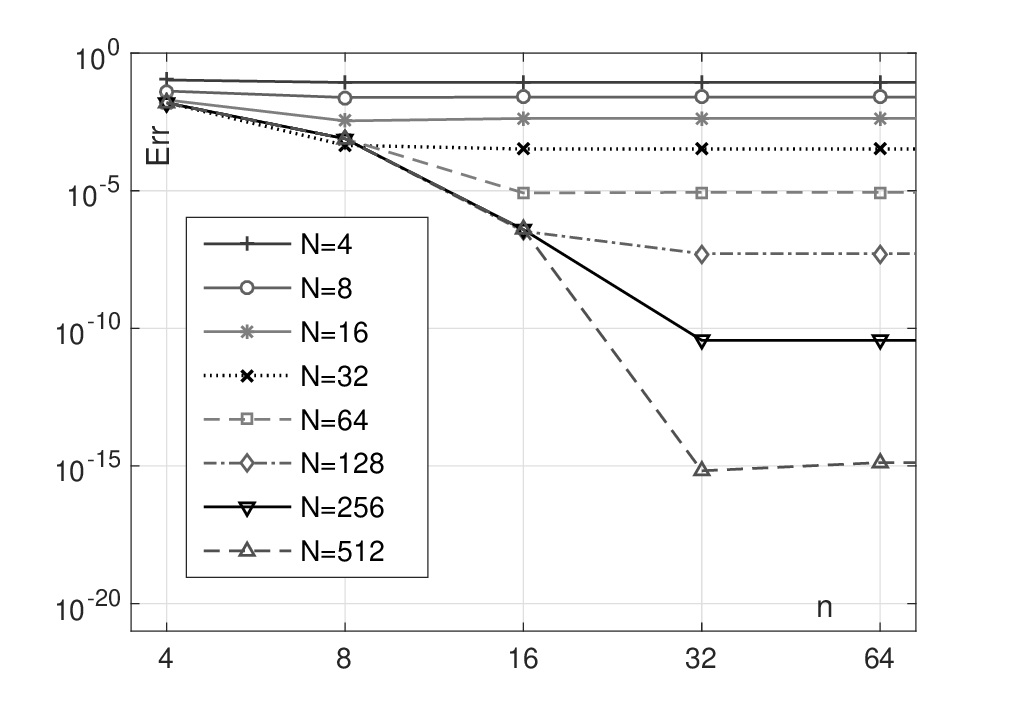}
		\put(42,74){\colorbox{white}{\small (a) $\mu=0.25$} }
		\end{overpic}
			\hspace{-8mm}
		\begin{overpic}[width=\lentwosubfig, viewport=20 20 465 380, clip=true]%
		{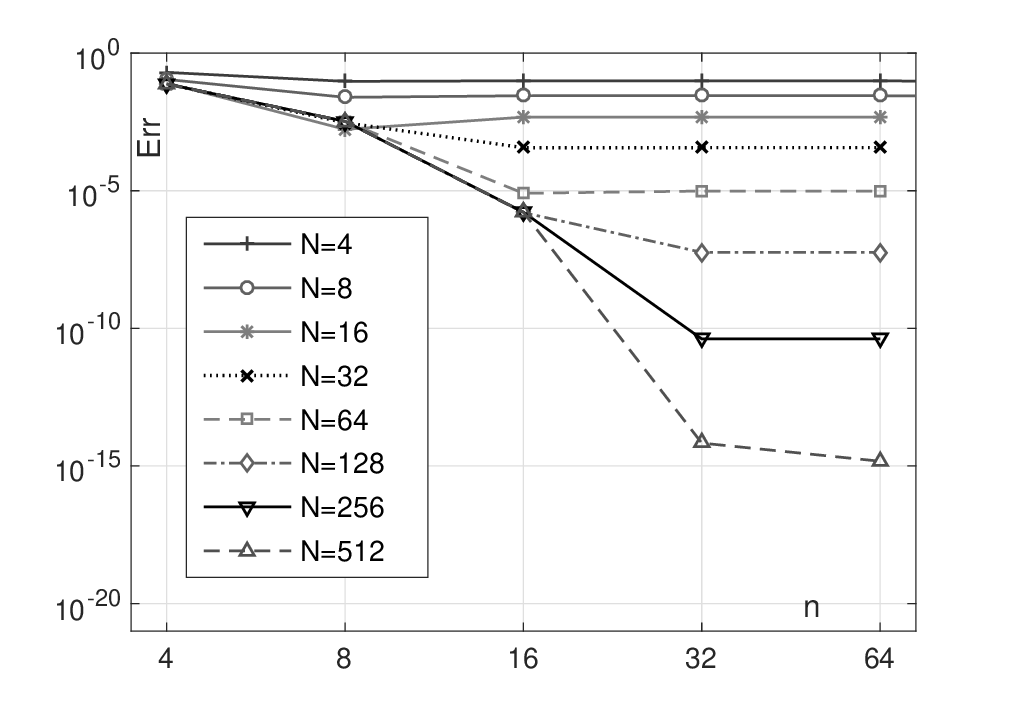}
		\put(42,74){\colorbox{white}{\small (b) $\mu=1$} }
		\end{overpic}
			\caption{Graph of experimental error Err as a function of the number of collocation points $n$, drawn in the logarithmic scale for $N=4,8,16,32,64,128,256,512$. }
			\label{fig:err_vs_n}
		\end{figure}
		Those two are shown in Table \ref{tab1} alongside the value of Err and the number of iterations (denoted by $K$) needed to achieve the iteration's relative accuracy $10^{-5}$, estimated a posteriori.
		As observed from Table \ref{tab1}, the experimental converge of the developed method is exponential with respect to $n$. 
		
		\begin{table}[ptbh]
			\begin{center}
				\begin{tabular}{|c|c|c|c|}
					\hline
					$N$                      & $n$  &                 Err                  & $K$  \\
					\hline
					4                      & 8  &        0.0859119243400000010         & 3  \\
					8                      & 8  &        0.0244950525900000000         & 3  \\
					16                      & 8  &        0.00345794666699999987        & 3  \\
					32                      & 16 &       0.000328787487900000005        & 3  \\
					64                      & 16 &      0.00000833843948899999922       & 3  \\
					128                     & 32 &     0.0000000515513076299999962      & 4  \\
					256                     & 32 & $3.68083566999999982\times 10^{-11}$ & 4 \\
					512                     & 64 & $1.32334447899999999\times 10^{-15}$ & 4 \\
					\hline
				\end{tabular}
			\end{center}
			\caption{Result of experimental application of the developed method to the numerical solution of \eqref{eq:CP-NN}}\label{tab1}.
		\end{table}
		The preposition of Theorem \ref{thm:CP-NN-sol_ex_ham_disc} is no longer true for $\mu = 1$. In spite of that the method is still convergent. Several graphs of $Err$ as a function of $n$ are depicted in Fig. \ref{fig:err_vs_n} (b). They demonstrate a qualitatively similar exponentially decreasing behaviour of the experimental error.
		
	\end{example}
	
	\section{Conclusions and future work}\label{sec:CC-NN-conclusions}
	In this work we developed and justified the method for the numerical solution of abstract nonlocal Cauchy problem \eqref{eq:CP-NN}. If the operator function $g\left (u(\cdot)\right )$ is continuous and $A$ is sectorial with the angle $\phi_0 < \frac{\pi}{2}$ is globally convergent. 
	Already being general, problem \eqref{eq:CP-NN} can be generalized even more by adding the nonlinear right-hand side. Existing theoretical results  for such problems \cite{AizicoviciLee2005,FanDongLi2006} suggest that  the extension of  developed methodology to that class of problems is possible. This consists the topic for the future work.
	
	\bibliographystyle{siam}
	\bibliography{bib-GMV,nonlocal_in_time}

\begin{thebibliography}{10}

\bibitem{AizicoviciLee2005}
{\sc S.~Aizicovici and H.~Lee}, {\em Nonlinear nonlocal cauchy problems in
  banach spaces}, Applied Mathematics Letters, 18 (2005), pp.~401 -- 407.

\bibitem{nonloc_int_time_Ashyralyev2014}
{\sc A.~Ashyralyev and N.~Aggez}, {\em Nonlocal boundary value hyperbolic
  problems involving integral conditions}, Boundary Value Problems, 2014
  (2014).

\bibitem{AshGer}
{\sc A.~Ashyralyev and O.~Gercek}, {\em Finite difference method for multipoint
  nonlocal elliptic-parabolic problems}, Comput. Math. Appl., 60 (2010),
  pp.~2043--2052.

\bibitem{NonlocalBitsadze1969}
{\sc A.~Bitsadze and A.~Samarskii}, {\em {On some simple generalizations of
  linear elliptic boundary problems.}}, Sov. Math., Dokl., 10 (1969),
  pp.~398--400.

\bibitem{Byszewski1991a}
{\sc L.~Byszewski}, {\em Theorems about the existence and uniqueness of
  solutions of a semilinear evolution nonlocal cauchy problem}, Journal of
  Mathematical Analysis and Applications, 162 (1991), pp.~494 -- 505.

\bibitem{Byszewski1992a}
\leavevmode\vrule height 2pt depth -1.6pt width 23pt, {\em Existence of
  approximate solution to abstract nonlocal cauchy problem}, Journal of Applied
  Mathematics and Stochastic Analysis, 5 (1992), pp.~363--373.

\bibitem{Byszewski1992}
\leavevmode\vrule height 2pt depth -1.6pt width 23pt, {\em Uniqueness of
  solutions of parabolic semilinear nonlocal-boundary problems}, Journal of
  Mathematical Analysis and Applications, 165 (1992), pp.~472 -- 478.

\bibitem{Cannon1986}
{\sc J.~R. Cannon and J.~van~der Hoek}, {\em Diffusion subject to the
  specification of mass}, Journal of Mathematical Analysis and Applications,
  115 (1986), pp.~517--529.

\bibitem{bClement1987}
{\sc P.~Cl{\'e}ment, H.~J. A.~M. Heijmans, S.~Angenent, C.~J. van Duijn, and
  B.~de~Pagter}, {\em One-parameter semigroups}, vol.~5 of CWI Monographs,
  North-Holland Publishing Co., Amsterdam, 1987.

\bibitem{FanDongLi2006}
{\sc Z.~Fan, Q.~Dong, and G.~Li}, {\em Semilinear differential equations with
  nonlocal conditions in banach spaces}, International Journal of Nonlinear
  Science, 2 (2006), pp.~131--139.

\bibitem{GavrilyukMakarovVasylyk2011}
{\sc I.~Gavrilyuk, V.~Makarov, and V.~Vasylyk}, {\em Exponentially convergent
  algorithms for abstract differential equations}, Frontiers in Mathematics,
  Birkh\"auser/Springer Basel AG, Basel, 2011.

\bibitem{Gavrilyuk2010}
{\sc I.~P. Gavrilyuk, V.~L. Makarov, D.~O. Sytnyk, and V.~B. Vasylyk}, {\em
  Exponentially convergent method for the m-point nonlocal problem for a first
  order differential equation in banach space}, {Numerical Functional Analysis
  and Optimization}, 31 (2010), pp.~1--21.

\bibitem{err_est_review_Goodenough1986}
{\sc S.~Goodenough}, {\em A link between lebesgue constants and
  hermite-fej{\'e}r interpolation}, Bulletin of the Australian Mathematical
  Society, 33 (1986), pp.~207--218.

\bibitem{nonloc_BS_Gordeziani1970}
{\sc D.~Gordeziani}, {\em A certain method of solving the
  {B}icadze-{S}amarski\u\i boundary value problem}, Gamoqeneb. Math. Inst. Sem.
  Mohsen. Anotacie, 2 (1970), pp.~39--41.

\bibitem{Gordeziani1981}
{\sc D.~G. Gordeziani}, {\em O metodakh resheniya odnogo klassa nelokalnykh
  kraevykh zadach},  (1981), p.~32.
\newblock With Georgian and English summaries.

\bibitem{Gordeziani1984}
\leavevmode\vrule height 2pt depth -1.6pt width 23pt, {\em A class of nonlocal
  boundary value problems in elasticity theory and shell theory}, in Theory and
  numerical methods of calculating plates and shells, {V}ol.\ {II} ({R}ussian)
  ({T}bilisi, 1984), Tbilis. Gos. Univ., Tbilisi, 1984, pp.~106--127.

\bibitem{NHE_Hess1971}
{\sc P.~Hess}, {\em On nonlinear equations of hammerstein type in banach
  spaces}, Proceedings of the American Mathematical Society, 30 (1971),
  pp.~308--312.

\bibitem{Cannon1987}
{\sc J.~V. D.~H. John R.~Cannon, Salvador Perez~Esteva}, {\em A galerkin
  procedure for the diffusion equation subject to the specification of mass},
  SIAM Journal on Numerical Analysis, 24 (1987), pp.~499--515.

\bibitem{MakarovKhlobystov1983}
{\sc V.~{Makarov} and V.~{Khlobystov}}, {\em {Spline - approximation of
  functions. Textbook. (Splajn - approksimatsiya funktsij. Uchebnoe
  posobie).}}, 1983.

\bibitem{nonloc_exMVS2014}
{\sc V.~L. Makarov, D.~O. Sytnyk, and V.~B. Vasylyk}, {\em Existence of the
  solution to a nonlocal-in-time evolutional problem}, {Nonlinear Analysis:
  Modelling and Control}, 19 (2014), pp.~432--447.

\bibitem{MolerLoan2003}
{\sc C.~Moler and C.~V. Loan}, {\em Nineteen dubious ways to compute the
  exponential of a matrix, twenty-five years later}, SIAM Review, 45 (2003),
  pp.~3--49.

\bibitem{Ntouyas2006461}
{\sc S.~Ntouyas}, {\em {Chapter 6 Nonlocal Initial and Boundary Value Problems:
  A Survey}}, vol.~2 of Handbook of Differential Equations: Ordinary
  Differential Equations, North-Holland, 2006, pp.~461 -- 557.

\bibitem{NtouyasTsamatos1997}
{\sc S.~K. Ntouyas and P.~C. Tsamatos}, {\em Global existence for semilinear
  evolution equations with nonlocal conditions}, Journal of Mathematical
  Analysis and Applications, 210 (1997), pp.~679 -- 687.

\bibitem{bPazy1983}
{\sc A.~Pazy}, {\em Semigroups of linear operator and applications to partial
  differential equations}, Springer Verlag, New York, Berlin, Heidelberg, 1983.

\bibitem{Smith2006}
{\sc S.~J. Smith}, {\em The lebesgue constants for polynomial interpolation},
  Annales Mathematicae et Informaticae,  (2006).

\bibitem{nonlocal_time_numerVabishchevich1982}
{\sc P.~Vabishchevich}, {\em {Nonlocal parabolic problems and the inverse
  heat-conduction problem.}}, Differ. Equations, 17 (1982), pp.~761--765.

\bibitem{vasNL_Ell}
{\sc V.~Vasylyk}, {\em Exponentially convergent method for the $m$-point
  nonlocal problem for an elliptic differential equation in banach space},
  Journal of Numerical and Applied Mathematics, 105 (2011), pp.~124--135.

\bibitem{VasIntEll}
\leavevmode\vrule height 2pt depth -1.6pt width 23pt, {\em Exponentially
  convergent method for integral nonlocal problem for the elliptic equa-tion in
  banach space}, Journal of Numerical and Ap-plied Mathematics, 110 (2013),
  pp.~119--130.

\bibitem{XiaoLiang2005}
{\sc T.-J. Xiao and J.~Liang}, {\em Existence of classical solutions to
  nonautonomous nonlocal parabolic problems}, Nonlinear Analysis: Theory,
  Methods \& Applications, 63 (2005), pp.~e225 -- e232.
\newblock Invited Talks from the Fourth World Congress of Nonlinear Analysts
  (WCNA 2004)Invited Talks from the Fourth World Congress of Nonlinear Analysts
  (WCNA 2004).

\bibitem{Zeidler1990}
{\sc E.~Zeidler}, {\em Monotone Operators and Hammerstein Integral Equations},
  Springer New York, New York, NY, 1990, pp.~615--638.

\end{thebibliography}

\end{document}